\documentclass[11pt,reqno]{amsart}
\usepackage{amsfonts,amssymb,amsmath,color}

\usepackage[pdftex]{graphicx}

\DeclareMathOperator{\bmin}{\mathbf{min}}

\DeclareMathOperator{\bmax}{\mathbf{max}}

\DeclareMathOperator{\precdot}{\prec\!\!\cdot}

\newtheorem{theorem}{Theorem}

\newtheorem{corollary}[theorem]{Corollary}

\newtheorem{example}[theorem]{Example}

\numberwithin{equation}{section} \numberwithin{theorem}{section}

\begin{document}

\title{Pattern Recognition on Oriented Matroids: Topes and Critical Committees}

\author{Andrey O. Matveev}
\email{andrey.o.matveev@gmail.com}


\begin{abstract}
Let the sign components of the maximal covectors of a simple oriented matroid $\mathcal{M}$ be represented by the real numbers $-1$ and $1$. Consider the vertex set $\mathfrak{V}(\boldsymbol{R})$ of a symmetric cycle $\boldsymbol{R}$ of adjacent topes in the tope graph of $\mathcal{M}$ as a subposet of the tope poset
of $\mathcal{M}$. If $B$ is the bottom element of the tope poset then $B$ is equal to the unweighted sum of the members of the set $\bmin\mathfrak{V}(\boldsymbol{R})$ of minimal elements of the subposet $\mathfrak{V}(\boldsymbol{R})$; if $B$ is the positive tope then the set $\bmin\mathfrak{V}(\boldsymbol{R})$ is a critical tope committee for the acyclic oriented matroid $\mathcal{M}$.
\end{abstract}

\maketitle

\pagestyle{myheadings}

\markboth{PATTERN RECOGNITION ON ORIENTED MATROIDS}{A.O.~MATVEEV}

\thispagestyle{empty}

\tableofcontents

\section{Introduction}

Let $\mathcal{M}:=(E_t,\mathcal{T})$ be an oriented matroid, of rank $\geq 2$, on the
ground set~$E_t$ $:=\{1,\ldots,t\}$, with set of topes $\mathcal{T}$; throughout we will suppose that it is {\em simple}, that is, it contains no loops, parallel or {\sl antiparallel\/} elements. In this paper the sign components $-$ and $+$ of maximal covectors are replaced by the real numbers $-1$ and $1$, respectively. The topes $T:=(T(1),\ldots,T(t))\in\mathcal{T}$ are interpreted as elements of the real Euclidean space
$\mathbb{R}^t$
of row vectors; if~$T,T',T''\in\mathbb{R}^t$ then
$\langle T',T''\rangle:=\sum_{e=1}^t T'(e)\cdot T''(e)$, and $\|T\|:=\sqrt{\langle T,T\rangle}$.
We\hfill denote\hfill by\hfill $\bigl(\boldsymbol{\sigma}(1),\ldots,\boldsymbol{\sigma}(t)\bigr)$\hfill the\hfill standard\hfill basis\hfill of\hfill $\mathbb{R}^t$,\hfill that\hfill is,\hfill
$\boldsymbol{\sigma}(i)$\\ $:=(0,\ldots,\underset{\overset{\uparrow}{i}}{1},\ldots,0)$, $1\leq i\leq t$. The {\em positive tope\/} $\mathrm{T}^{(+)}$ is the vector $(1,\ldots,1)$ $=\sum_{i=1}^t\boldsymbol{\sigma}(i)$.

If $e\in E_t$ then the corresponding {\em negative\/} and {\em positive halfspaces\/} are the tope subsets~$\mathcal{T}_e^-:=\{T\in\mathcal{T}:\ T(e)=-1\}$ and $\mathcal{T}_e^+:=\{T\in\mathcal{T}:\ T(e)=1\}$, respectively.

If $T\in\mathcal{T}$ and $A\subseteq E_t$ then ${}_{-A}T$ denotes the vector obtained from $T$ by {\em sign reversal} or {\em reorientation\/} on the set $A$: $({}_{-A}T)(e)=-T(e)$ when $e\in A$, and $({}_{-A}T)(e)= T(e)$ when $e\not\in A$. The oriented matroid whose topes are obtained from the topes of $\mathcal{M}$ by reorientation on the set $A$ is denoted by~${}_{-A}\mathcal{M}$.

We denote by $T^-$
the {\em negative part\/} $\{e\in E_t:\ T(e)=-1\}$
of the tope $T$; the {\em positive part\/} $T^+$ of $T$ is the set $\{e\in E_t:\ T(e)=1\}$.

The vertices of the {\em tope graph\/} $\mathcal{T}(\mathcal{L}(\mathcal{M}))$ of the oriented matroid $\mathcal{M}$ are its topes; a pair of topes $\{T',T''\}\subset\mathcal{T}$ is an edge of the graph $\mathcal{T}(\mathcal{L}(\mathcal{M}))$ if the topes $T'$ and $T''$ are {\em adjacent}, that is, they cover some {\em subtope\/} in the {\em
big face lattice\/} of $\mathcal{M}$.
The {\em separation set\/} $\mathbf{S}(T',T'')$ of topes $T'$ and $T''$ is defined by
$\mathbf{S}(T',T''):=\{e\in E_t:\ T'(e)\neq T''(e)\}$. The {\em graph distance\/}~$d(T',T'')$ between the topes $T'$ and $T''$ is the cardinality of the separation set $\mathbf{S}(T',T'')$, see~\cite[Prop.~4.2.3]{BLSWZ}, that is,
\begin{equation*}
d(T',T'')=|\mathbf{S}(T',T'')|=t-\tfrac{1}{4}\|T''+T'\|^2=\tfrac{1}{4}\|T''-T'\|^2=
\tfrac{1}{2}\!\left(t-\langle T'',T'\rangle\right)\; .
\end{equation*}

If $B\in\mathcal{T}$ then we denote by~$\mathcal{T}(\mathcal{L}(\mathcal{M}),B)$ the {\em tope poset\/} of $\mathcal{M}$ based at the tope $B$; by convention, we have $T'\preceq T''$ in $\mathcal{T}(\mathcal{L}(\mathcal{M}),B)$ iff $\mathbf{S}(B,T')$ $\subseteq\mathbf{S}(B,T'')$. If~$\mathcal{X}\subset\mathcal{T}(\mathcal{L}(\mathcal{M}),B)$ then $\bmin\mathcal{X}$
stands for the set of minimal elements of the subposet $\mathcal{X}$.

Let $\mathbf{m}:=(R^0:=B \precdot R^1\precdot\cdots\precdot R^{t-1}\precdot R^t:=-B)$ be a maximal chain in the tope poset
$\mathcal{T}(\mathcal{L}(\mathcal{M}),B)$, and $-\mathbf{m}:=\{-R:\ R\in\mathbf{m}\}$. The union~$\mathfrak{V}(\boldsymbol{R})$\hfill $:=$\hfill $\mathbf{m}\cup-\mathbf{m}$\hfill is\hfill the\hfill vertex\hfill set\hfill of\hfill the\hfill {\em symmetric\hfill cycle}\hfill $\boldsymbol{R}$\\ $:=(R^0:=B,R^1,\ldots,R^{2t-1},R^0)$ in the tope graph $\mathcal{T}(\mathcal{L}(\mathcal{M}))$; by convention, we have $R^{k+t}=-R^k$, $0\leq k\leq t-1$. The subset of topes from $\mathfrak{V}(\boldsymbol{R})$ with inclusion-maximal positive parts is denoted by $\bmax^+(\mathfrak{V}(\boldsymbol{R}))$. If $A\subseteq E_t$ then ${}_{-A}\mathfrak{V}(\boldsymbol{R}):=
\{{}_{-A}R:\ R\in\mathfrak{V}(\boldsymbol{R})\}$.

Let $(\mathfrak{l}_1,\ldots,\mathfrak{l}_t)\in\mathbb{N}^t$ be the sequence defined by
$\{\mathfrak{l}_i\}:=\mathbf{S}(R^{i-1},R^i)$; note that 
$\mathbf{m}-\{-B\}\subseteq\mathcal{T}^-_{\mathfrak{l}_t}$ if $B(\mathfrak{l}_t)=-1$, and $\mathbf{m}-\{-B\}\subseteq\mathcal{T}^+_{\mathfrak{l}_t}$ if $B(\mathfrak{l}_t)=1$.

The chain $\mathbf{m}-\{-B\}$ is a basis of the space $\mathbb{R}^t$; indeed, the square {\em sign matrix}
\begin{equation*}\mathbf{M}:=\mathbf{M}(\boldsymbol{R}):=
\left(
\begin{smallmatrix}
R^0\\R^1\\ \vdots\\ R^{t-2}\\ R^{t-1}
\end{smallmatrix}
\right)\!\in\mathbb{R}^{t\times t}
\end{equation*}
is similar to the nonsingular matrix
\begin{equation*}
\left(
\begin{smallmatrix}
2\cdot B(\mathfrak{l}_1)\cdot\boldsymbol{\sigma}(\mathfrak{l}_1)\\ 2\cdot B(\mathfrak{l}_2)\cdot\boldsymbol{\sigma}(\mathfrak{l}_2)\\ \vdots\\ 2\cdot B(\mathfrak{l}_{t-1})\cdot\boldsymbol{\sigma}(\mathfrak{l}_{t-1})\\ B(\mathfrak{l}_t)\cdot\boldsymbol{\sigma}(\mathfrak{l}_t)
\end{smallmatrix}
\right)\; ;
\end{equation*}
the absolute value of its determinant is $2^{t-1}$.

The $i$th row $(\mathbf{M}^{-1})_i$, $1\leq i\leq t$, of the inverse matrix $\mathbf{M}^{-1}$ of $\mathbf{M}$ is
\begin{equation*}
(\mathbf{M}^{-1})_i=\begin{cases}
\frac{1}{2}\cdot B(i)\cdot(\;\boldsymbol{\sigma}(k)-\boldsymbol{\sigma}(k+1)\;),&\text{if
$i=\mathfrak{l}_k$, $k\neq t$}\; ,\\
\frac{1}{2}\cdot B(i)\cdot(\;\boldsymbol{\sigma}(1)+\boldsymbol{\sigma}(t)\;),&\text{if $i=\mathfrak{l}_t$}\; .
\end{cases}
\end{equation*}
Thus, if $T\in\mathcal{T}$ then we have $T=\boldsymbol{x}\mathbf{M}$
for some row vector
\begin{equation}
\label{eq:4}
\boldsymbol{x}:=(x_1,\ldots,x_t)=T\mathbf{M}^{-1}
\end{equation}
such that
\begin{equation*}
\boldsymbol{x}\in\{-1,0,1\}^t\; .
\end{equation*}

A subset $\mathcal{K}^{\ast}\subset\mathcal{T}$ is called a {\em tope committee\/} for $\mathcal{M}$ if
\begin{equation*}
\sum_{T\in\mathcal{K}^{\ast}}T\geq\mathrm{T}^{(+)}\; ,
\end{equation*}
see~\cite{M-Halfspaces,M-Layers,M-Existence,M-Three,M-Reorientations}.
The committee $\mathcal{K}^{\ast}$ is called {\em minimal\/} if any its proper subset is not a committee for $\mathcal{M}$. If the sum of the members of the minimal tope committee $\mathcal{K}^{\ast}$ is the
positive tope,
\begin{equation*}
\sum_{T\in\mathcal{K}^{\ast}}T=\mathrm{T}^{(+)}\; ,
\end{equation*}
then we say that the committee $\mathcal{K}^{\ast}$ is {\em critical\/}.

Recall that if $\boldsymbol{R}$ is a symmetric cycle in the tope graph $\mathcal{T}(\mathcal{L}(\mathcal{M}))$ of the oriented matroid $\mathcal{M}$, then
for any tope $T\in\mathcal{T}$ there exists a unique inclusion-minimal subset~$\boldsymbol{Q}(T,\boldsymbol{R})\subset\mathfrak{V}(\boldsymbol{R})$ such that
\begin{equation}
\label{eq:2}
T=\sum_{Q\in\boldsymbol{Q}(T,\boldsymbol{R})}Q\; ;
\end{equation}
this set, of odd cardinality, of linearly independent elements of $\mathbb{R}^t$ is
\begin{equation*}
\boldsymbol{Q}(T,\boldsymbol{R})=\{x_i\cdot R^{i-1}:\ x_i\neq 0\}\; ,
\end{equation*}
where the vector $\boldsymbol{x}$ is defined by~(\ref{eq:4}).

As a consequence,
a tope subset $\mathcal{K}^{\ast}\subset\mathcal{T}$ is a committee for $\mathcal{M}$ iff
\begin{equation*}
\sum_{T\in\mathcal{K}^{\ast}}\ \sum_{Q\in\boldsymbol{Q}(T,\boldsymbol{R})}Q\geq\mathrm{T}^{(+)}\; .
\end{equation*}

In Section~\ref{section:1} we discuss relation~(\ref{eq:2}) and describe the structure of the sets $\boldsymbol{Q}(T,\boldsymbol{R})$ in more detail.

\section{Topes and Critical Committees}
\label{section:1}

For a tope $T$ of the oriented matroid $\mathcal{M}$ and for a symmetric cycle $\boldsymbol{R}$ in its tope graph,
we describe the corresponding inclusion-minimal subset~$\boldsymbol{Q}(T,\boldsymbol{R})\subset\mathfrak{V}(\boldsymbol{R})$ from expression~(\ref{eq:2}) in terms of the tope poset of~$\mathcal{M}$. Dual assertions could be made because the tope poset of $\mathcal{M}$ is self-dual in the sense of~\cite[Prop.~4.2.15(ii)]{BLSWZ}.

\begin{theorem}
Let $\boldsymbol{R}:=(R^0,R^1,\ldots,R^{2t-1},R^0)$ be a symmetric cycle in the tope graph $\mathcal{T}(\mathcal{L}(\mathcal{M}))$ of the oriented matroid $\mathcal{M}$. Pick a tope $B\in\mathcal{T}$ and consider the vertex set $\mathfrak{V}(\boldsymbol{R})$ of the cycle $\boldsymbol{R}$ as a subposet of the
tope poset~$\mathcal{T}(\mathcal{L}(\mathcal{M}),B)$ based at $B$.

The set $\boldsymbol{Q}(B,\boldsymbol{R})$ is the set $\bmin\mathfrak{V}(\boldsymbol{R})$ of minimal elements of the subposet~$\mathfrak{V}(\boldsymbol{R})\subset\mathcal{T}(\mathcal{L}(\mathcal{M}),B)$; therefore
\begin{equation}
\label{eq:1}
B=\sum_{Q\in\bmin\mathfrak{V}(\boldsymbol{R})}Q\; .
\end{equation}
\end{theorem}

\begin{proof}
If $B\in\mathfrak{V}(\boldsymbol{R})$, there is nothing to prove. Suppose that $B\not\in\mathfrak{V}(\boldsymbol{R})$, and reorient the items of the poset $\mathcal{T}(\mathcal{L}(\mathcal{M}),B)$ on the
negative part $B^-$ of the tope $B$; in other words, consider the tope poset $\mathcal{T}\bigl(\mathcal{L}({}_{-(B^-)}\mathcal{M}),\mathrm{T}^{(+)}\bigr)$ of the acyclic oriented matroid ${}_{-(B^-)}\mathcal{M}$ based at the positive tope $\mathrm{T}^{(+)}$.

If\hfill $O$\hfill is\hfill a\hfill tope\hfill of\hfill the\hfill oriented\hfill matroid\hfill ${}_{-(B^-)}\mathcal{M}$,\hfill then\hfill the\hfill poset\\ rank~$d(\mathrm{T}^{(+)},O)$ of $O$ in the tope poset $\mathcal{T}\bigl(\mathcal{L}({}_{-(B^-)}\mathcal{M}),\mathrm{T}^{(+)}\bigr)$ of ${}_{-(B^-)}\mathcal{M}$ is equal to the cardinality $|O^-|=|\mathbf{S}(\mathrm{T}^{(+)},O)|$ of the negative part of $O$.

A\hfill tope\hfill $O$\hfill of\hfill the\hfill oriented\hfill matroid\hfill ${}_{-(B^-)}\mathcal{M}$\hfill belongs\hfill to\hfill the\\ set\hfill $\bmax^+\bigl({}_{-(B^-)}\mathfrak{V}(\boldsymbol{R})\bigr)$\hfill iff\hfill
for\hfill the\hfill $2$-path\hfill $(O',O,O'')$,\hfill where\hfill $O'\neq O''$,\hfill in\hfill
the\\ symmetric\hfill cycle\hfill $\bigl({}_{-(B^-)}R^0,$\hfill ${}_{-(B^-)}R^1,$\hfill $\ldots,$\hfill ${}_{-(B^-)}R^{2t-1},$\hfill ${}_{-(B^-)}R^0\bigr)$\hfill we\\ have~$d(\mathrm{T}^{(+)},O')=d(\mathrm{T}^{(+)},O'')=d(\mathrm{T}^{(+)},O)+1$. By~\cite[Prop.~5.6]{M-Existence}, the set~$\bmax^+\bigl({}_{-(B^-)}\mathfrak{V}(\boldsymbol{R})\bigr)=\bmin{}_{-(B^-)}\mathfrak{V}(\boldsymbol{R})$ of minimal elements of the subposet ${}_{-(B^-)}\mathfrak{V}(\boldsymbol{R})\subset\mathcal{T}\bigl(\mathcal{L}({}_{-(B^-)}\mathcal{M}),
\mathrm{T}^{(+)}\bigr)$ is the inclusion-minimal subset of topes with the property $\sum_{T\in\bmax^+({}_{-(B^-)}\mathfrak{V}(\boldsymbol{R}))}T$ $=\mathrm{T}^{(+)}$, that is,
\begin{equation}
\label{eq:3}
\mathrm{T}^{(+)}=\sum_{T\in\bmin{}_{-(B^-)}\mathfrak{V}(\boldsymbol{R})}T\; .
\end{equation}
This means that $\bmin{}_{-(B^-)}\mathfrak{V}(\boldsymbol{R})$ is a critical tope committee for the acyclic oriented matroid ${}_{-(B^-)}\mathcal{M}$. Relation~(\ref{eq:1}) is equivalent to~(\ref{eq:3}).
\end{proof}

\begin{corollary} If $\boldsymbol{R}$ is a symmetric cycle in the tope graph of the oriented matroid~$\mathcal{M}$, then for any tope $T\in\mathcal{T}$ the corresponding set~$\boldsymbol{Q}(T,\boldsymbol{R})$ is the set
${}_{-(T^-)}\Bigl(\bmax^+\bigl({}_{-(T^-)}\mathfrak{V}(\boldsymbol{R})\bigr)\Bigr)$; therefore
\begin{equation*}
T=\sum_{Q\ \in\ {}_{-(T^-)}\left(\bmax^+\left({}_{-(T^-)}\mathfrak{V}(\boldsymbol{R})\right)\right)}Q\; .
\end{equation*}
\end{corollary}

\begin{example}
Consider the Hasse diagram, which is depicted in Figure~{\rm\ref{figure:1}(a)}, of the tope poset $\mathcal{T}(\mathcal{L}(\mathcal{M}),B)$ of a simple oriented matroid $\mathcal{M}$ of rank $3$, where $B:=(-1,1,1,1,1)$. Fix the symmetric cycle
$\boldsymbol{R}$ $:=$ $\bigl((1,-1,1,1,1),$ $(1,-1,1,-1,1),$\hfill $(1,-1,-1,-1,1),$\hfill $(1,1,-1,-1,1),$\hfill
$(-1,1,-1,-1,1),$\\ $(-1,1,-1,-1,-1),$\hfill $(-1,1,-1,1,-1),$\hfill $(-1,1,1,1,-1),$\hfill $(-1,-1,1,1,-1),$\\ $(1,-1,1,1,-1),$ $(1,-1,1,1,1)\bigr)$
in the tope graph of $\mathcal{M}$.

The\hfill Hasse\hfill diagram,\hfill borrowed\hfill from\hfill \cite[Fig.~4.2.2]{BLSWZ},\hfill of\hfill the\hfill tope\\ poset~$\mathcal{T}\bigl(\mathcal{L}({}_{-\{1\}}\mathcal{M}),\mathrm{T}^{(+)}\bigr)$ of the acyclic oriented matroid ${}_{-\{1\}}\mathcal{M}$ is shown in~Figure~{\rm\ref{figure:1}(b)}. We have
$\bmin{}_{-(B^-)}\mathfrak{V}(\boldsymbol{R})=\bigl\{(-1,-1,1,1,1),(1,1,-1,-1,1),$ $(1,1,1,1,-1)\bigr\}$ and
\begin{equation*}
\mathrm{T}^{(+)}=\sum_{Q\in\bmin{}_{-(B^-)}\mathfrak{V}(\boldsymbol{R})}Q=(-1,-1,1,1,1)\;+\; (1,1,-1,-1,1)\;+\; (1,1,1,1,-1)\; .
\end{equation*}
Similarly, $\bmin\mathfrak{V}(\boldsymbol{R})=\bigl\{(1,-1,1,1,1),(-1,1,-1,-1,1),(-1,1,1,1,-1)\bigr\}$ and
\begin{equation*}
\begin{split}
B&=\sum_{Q\in\bmin\mathfrak{V}(\boldsymbol{R})}Q=(1,-1,1,1,1)\;+\; (-1,1,-1,-1,1)\;+\; (-1,1,1,1,-1)\\&=(-1,1,1,1,1)\; .
\end{split}
\end{equation*}

\begin{figure}
\includegraphics[viewport=6mm 0mm 140mm 60mm]{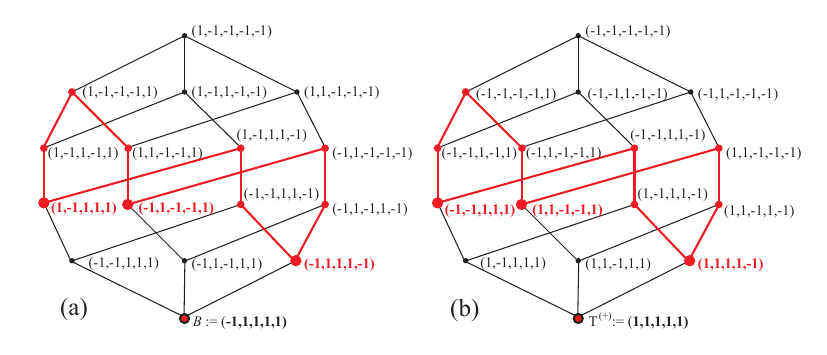}
\caption{(a): The Hasse diagram of the tope poset~$\mathcal{T}(\mathcal{L}(\mathcal{M}),B)$ based at the tope $B:=(-1,1,1,1,1)$, and the \textcolor{red}{subposet} $\mathfrak{V}(\boldsymbol{R})$ which is the vertex set
of a symmetric cycle
$\boldsymbol{R}$
in the tope graph.
(b): The Hasse diagram of the tope poset $\mathcal{T}\bigl(\mathcal{L}({}_{-\{1\}}\mathcal{M}),\mathrm{T}^{(+)}\bigr)$ based at the positive tope $\mathrm{T}^{(+)}:=(1,1,1,1,1)$.}
\label{figure:1}
\end{figure}

\begin{figure}
\includegraphics[viewport=6mm 0mm 140mm 60mm]{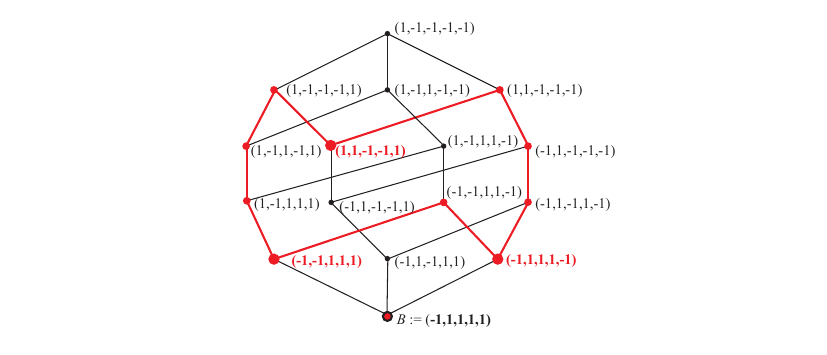}
\caption{The Hasse diagram of the tope poset~$\mathcal{T}(\mathcal{L}(\mathcal{M}),B)$ based at the tope $B:=(-1,1,1,1,1)$, and the \textcolor{red}{subposet} $\mathfrak{V}(\boldsymbol{R})$ which is the vertex set
of a symmetric cycle $\boldsymbol{R}$ in the tope graph, cf.~Figure~\ref{figure:1}(a).}
\label{figure:2}
\end{figure}

\begin{figure}
\includegraphics[viewport=6mm 0mm 140mm 60mm]{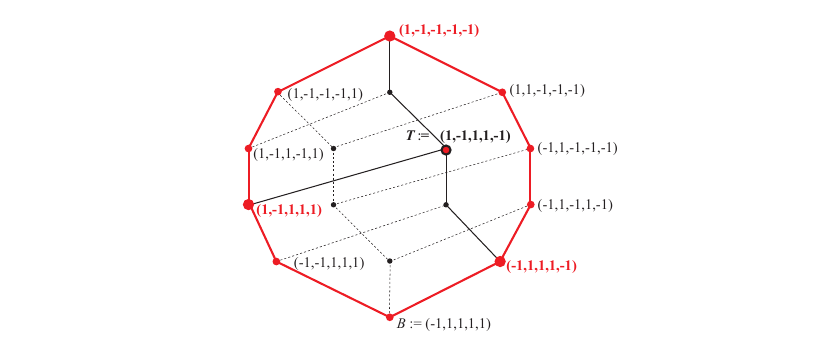}
\caption{The Hasse diagram of the tope poset~$\mathcal{T}(\mathcal{L}(\mathcal{M}),B)$ based at the tope $B:=(-1,1,1,1,1)$, and the \textcolor{red}{subposet} $\mathfrak{V}(\boldsymbol{R})$ which is the vertex set
of a symmetric cycle $\boldsymbol{R}$ in the tope graph, cf.~Figures~\ref{figure:1}(a) and~\ref{figure:2}.}
\label{figure:3}
\end{figure}

The subposet of the poset $\mathcal{T}(\mathcal{L}(\mathcal{M}),B)$, depicted in Figure~{\rm\ref{figure:2}}, is the vertex set of another symmetric cycle $\boldsymbol{R}:=\bigl((-1,-1,1,1,1),(1,-1,1,1,1),$ $(1,-1,1,-1,1),$\hfill $(1,-1,-1,-1,1),$\hfill $(1,1,-1,-1,1),$\hfill
$(1,1,-1,-1,-1),$ $(-1,1,-1,-1,-1),$\hfill $(-1,1,-1,1,-1),$\hfill
$(-1,1,1,1,-1),$\hfill $(-1,-1,1,1,-1),$ \\$(-1,-1,1,1,1)\bigr)$\hfill in\hfill the\hfill tope\hfill graph\hfill of\hfill $\mathcal{M}$.\hfill We\hfill have\hfill
$\bmin\mathfrak{V}(\boldsymbol{R})$ $=\bigl\{(-1,-1,1,1,1),(1,1,-1,-1,1),(-1,1,1,1,-1)\bigr\}$ and
\begin{equation*}
\begin{split}
B&=\sum_{Q\in\bmin\mathfrak{V}(\boldsymbol{R})}Q=(-1,-1,1,1,1)\;+\; (1,1,-1,-1,1)\;+\; (-1,1,1,1,-1)\\&= (-1,1,1,1,1)\; .
\end{split}
\end{equation*}


Let $\boldsymbol{R}:=\bigl((-1,1,1,1,1),(-1,-1,1,1,1),(1,-1,1,1,1),(1,-1,1,-1,1),$ $(1,-1,-1,-1,1),$\hfill
$(1,-1,-1,-1,-1),$\hfill $(1,1,-1,-1,-1),$\hfill $(-1,1,-1,-1,-1),$\\ $(-1,1,-1,1,-1),(-1,1,1,1,-1),(-1,1,1,1,1)\bigr)$ be one more symmetric cycle\hfill in\hfill the\hfill tope\hfill graph\hfill of\hfill the\hfill oriented\hfill  matroid\hfill  $\mathcal{M}$,\hfill see\hfill Figure\hfill {\rm\ref{figure:3}}.\hfill Pick\\ the\hfill tope\hfill $T:=(1,-1,1,1,-1)$\hfill of\hfill $\mathcal{M}$.\hfill
Since\hfill
${}_{-(T^-)}\Bigl(\bmax^+\bigl({}_{-(T^-)}\mathfrak{V}(\boldsymbol{R})\bigr)\Bigr)$\\ $=
\bigl\{(1,-1,1,1,1),(-1,1,1,1,-1),(1,-1,-1,-1,-1)\bigr\}$, we have
\begin{equation*}
\begin{split}
T:&=(1,-1,1,1,-1)=\sum_{Q\ \in\ {}_{-(T^-)}\left(\bmax^+\left({}_{-(T^-)}\mathfrak{V}(\boldsymbol{R})\right)\right)}Q\\&=(1,-1,1,1,1)\;
+\;(-1,1,1,1,-1)\;+\;(1,-1,-1,-1,-1)\; .
\end{split}
\end{equation*}
\end{example}

\end{document}